\documentclass[11pt]{amsart}
\usepackage{amsmath, amssymb, amsthm, amsfonts, mathrsfs}
\usepackage{xcolor}

\newtheorem{theorem}{Theorem}[section]

\newtheorem{corollary}[theorem]{Corollary}
\newtheorem{proposition}[theorem]{Proposition}
\newtheorem{lemma}[theorem]{Lemma}

\newtheorem{question}[theorem]{Question}

\theoremstyle{definition}
\newtheorem{definition}[theorem]{Definition}
\newtheorem{notation}[theorem]{Notation}
\newtheorem{remark}[theorem]{Remark}

\newtheorem{example}[theorem]{Example}

\numberwithin{equation}{section}

\DeclareSymbolFont{AMSb}{U}{msb}{m}{n}
\DeclareMathSymbol{\N}{\mathbin}{AMSb}{"4E}
\DeclareMathSymbol{\Z}{\mathbin}{AMSb}{"5A}
\DeclareMathSymbol{\R}{\mathbin}{AMSb}{"52}
\DeclareMathSymbol{\Q}{\mathbin}{AMSb}{"51}
\DeclareMathSymbol{\I}{\mathbin}{AMSb}{"49}
\DeclareMathSymbol{\C}{\mathbin}{AMSb}{"43}

\theoremstyle{definition}

\newcommand{\im}{{\rm im}}

\newcommand{\cU}{{\mathcal U}}

\newcommand{\Nb}{{\mathbb N}}

\newcommand{\Qb}{{\mathbb Q}}

\newcommand{\diam}{{\rm diam}}

\newcommand{\mesh}{{\rm mesh}}

\newcommand{\Hom}{{\rm Hom}}

\allowdisplaybreaks

\begin{document}

\title{On the denseness of distal points}
\author{Su Gao}
\address{School of Mathematical Sciences and LPMC\\ Nankai University \\ Tianjin 300071 \\P.R. China}
\email{sgao@nankai.edu.cn}
\thanks{S. Gao acknowledges the partial support of his research by the Fundamental Research Funds for the Central Universities and by the National Science Foundation of China (NSFC) grant 12271263.}
\author{Hanfeng Li}
\address{College of Mathematics and Statistics, Center of Mathematics, Chongqing University, Chongqing 401331, P.R. China}
\email{hfli@cqu.edu.cn}
\author{Ruiwen Li}
\address{School of Mathematical Sciences and LPMC\\ Nankai University \\ Tianjin 300071 \\P.R. China}
\email{rwli@mail.nankai.edu.cn}
\date{\today}
\begin{abstract} We give an answer to a question of Xu and Ye \cite{XY} on the denseness of distal points in the Bernoulli shift $2^G$ for a countable discrete group $G$. For a related but stronger notion of almost automorphic points, we answer the similar question by showing that the corresponding collection of the groups coincides with maximal almost periodic ones. These characterizations allow us to construct $2$-step nilpotent groups for which the answers to the Xu-Ye question differ. In search for an intrinsic answer to the Xu-Ye question, we introduce a notion of point-distal radical for a countable discrete group and show that a necessary condition is for the point-distal radical to be trivial. Finally, we consider some related questions, and show that the collection of all countable groups $G$ for which the set of distal points is dense in $2^G$ is closed under finite-index extension, and that the collection of countable groups $G$ for which the constant sequences are the only distal (almost automorphic) points coincides with the minimally almost periodic ones.
\end{abstract}
\maketitle

\section{Introduction}
The research in this paper was motivated by the following question asked by Xu and Ye:

\begin{question}[{\cite[Question (2)]{XY}}]\label{XYQ} For which countably infinite group $G$, the set of all distal points is dense in $(2^G, G)$?
\end{question}

As shown in recent research \cite{XY, HSXY}, this question is closely related to the question of which $G$-systems are disjoint with all mimimal $G$-systems. It was shown in \cite{XY} that the answer to this question is no for the Tarski monster group and yes for all residually finite groups.

In this paper we first consider the notion of almost automorphic points, which is stronger than distality. We will consider the following similar question. 

\begin{question}\label{AAP} For which countably infinite group $G$, the set of all almost automorphic points is dense in $(2^G, G)$?
\end{question}

It turns out that a positive answer to Question~\ref{AAP} is given by a well-studied notion of topological groups.

\begin{theorem}\label{thm:main1} For a countably infinite group $G$, the set of all almost automorphic points is dense in $(2^G, G)$ if and only if, as a discrete group, $G$ is maximally almost periodic.
\end{theorem}

Thus for all countable maximally almost periodic groups $G$, the set of all distal points is also dense in $(2^G, G)$. The notion of maximal almost periodicity (also known as MAP) was introduced by von Neumann \cite{vN}. It is well known that all countable abelian groups and residually finite groups are MAP. 

Next we turn to Question~\ref{XYQ}. We will characterize countably infinite discrete groups $G$ for which Question~\ref{XYQ} has a positive answer.

\begin{theorem}\label{thm:main2} For a countably infinite group $G$, the set of all distal points is dense in $(2^G, G)$ if and only if $G$ admits an effective point-distal continuous action on some compact metrizable space.
\end{theorem}

The characterization in Theorem~\ref{thm:main2} is not intrinsic, and thus we continue to search for a possible intrinsic criterion. For this we introduce a notion of point-disal radicals for countable groups, and show the following result.

\begin{theorem}\label{thm:main3} For a countable discrete group $G$, if the set of all distal points is dense in $(2^G, G)$, then the point-distal radical of $G$ is trivial.
\end{theorem}

Unfortunately we do not know yet if the converse is true. 

Using these criteria, we construct two examples of countable $2$-step nilpotent groups $G_1$ and $G_2$, neither of which is MAP, so that the set of distal points is dense in $(2^{G_1}, G_1)$ and not dense in $(2^{G_2}, G_2)$. This suggests that there is a very fine line between groups for which Question~\ref{XYQ} has positive and negative answers. In particular, there is unlikely a characterization of the groups in algebraic terms.

For any countably infinite group $G$, the constant sequences $\overline{0}$ and $\overline{1}$ are almost automorphic points (and therefore distal points) in  $2^G$. We refer to them as {\em trivial} distal (or almost automorphic) points. We also consider the following question which is opposite to Questions~\ref{XYQ} and \ref{AAP}.

\begin{question}\label{Trivial} For which countably infinite group $G$, $\overline{0}$ and $\overline{1}$ are the only distal (or almost automorphic) points in $(2^G, G)$?
\end{question}

We have the following characterization, which is related to a notion also introduced by von Neumann \cite{vN}.

\begin{theorem}\label{thm:main4} For a countable discrete group $G$, the following are equivalent:
\begin{enumerate}
\item $\overline{0}$ and $\overline{1}$ are the only distal points in $(2^G, G)$;
\item $\overline{0}$ and $\overline{1}$ are the only almost automorphic points in $(2^G, G)$;
\item $G$ is minimally almost periodic (minAP).
\end{enumerate}
\end{theorem}

It turns out that the Tarski monster group is minAP, along with many other examples such as $\mbox{Alt}(\mathbb{N})$, Hall's group $\mathbb{H}$, and $SL(n,K)$ and $PSL(n, K)$ for $n\geq 2$ and a countably infinite field $K$. 

Finally, we note that the collection of countable groups $G$ for which the set of distal points is dense in $(2^G, G)$ is closed under taking subgroups and finite-index extensions. 

The rest of this paper is organized as follows. In Section~\ref{sec:2} we characterize distal points in terms of IP*-recurrence and in terms of a new notion of TIP*-sets. In Section~\ref{sec:3} we frist characterize almost automorphic points in terms of $\Delta^*$-recurrence and T$\Delta^*$-sets, and then prove Theorem~\ref{thm:main1}. In Section~\ref{sec:4} we prove Theorem~\ref{thm:main2} and give a number of other characterizations by considering continuous actions of $G$ on compact metrizable spaces. In Section~\ref{sec:6} we introduce the notion of point-distal radical and prove Theorem~\ref{thm:main3}. Finally, in Section~\ref{sec:5} we prove some related results, including the closure under finite-index extension and Theorem~\ref{thm:main4}. 

{\em Acknowledgments.} We would like to thank Xiangdong Ye for bringing Question~\ref{XYQ} to our attention and for discussions on \cite{XY}.

\section{Distal Points and IP*-Recurrence\label{sec:2}}

In this section we recall some basic concepts around the notion of distality. Following Furstenburg \cite{Fur}, we characterize distal points in terms of IP*-recurrence for arbitrary countably infinite groups. Furthermore, we introduce a notion of TIP*-sets for countable groups and show that it corresponds to distal points in the Bernoulli shift system. 

\subsection{Distal points}\hfill\ 

All concepts and results without explicit references can be found in \cite{Fur}.

 A {\em dynamical system} consists of a compact metric space $X$ together with a group or semigroup $G$ acting on $X$ by continuous transformations. A dynamical system $(X,G)$ is {\em distal} if whenever $x\neq y\in X$ there is $\epsilon=\epsilon(x,y)>0$ such that for all $g\in G$, $d(gx, gy)>\epsilon$. Two points $x, y\in X$ are {\em proximal} if for some sequence $\{g_n\}$ in $G$, we have $d(g_nx,g_ny)\to 0$; they are {\em distal} if they are not proximal.

Note that for a dynamical system, being distal is a topological property, meaning that if $d'$ is another compatible metric on $X$, then $(X, G)$ is distal with respect to $d$ if and only if $(X,G)$ is distal with respect to $d'$.

\begin{definition}[Following \cite{DLX}] Let $G$ be a topological semigroup and $S\subseteq G$.
\begin{enumerate}
\item $S$ is {\em (left) syndetic} if there is a nonempty compact subset $K$ of $G$ such that for any $g\in G$ there is $k\in K$ such that $kg\in S$.
\item $S$ is {\em thick} if for any nonempty compact subset $K$ of $G$ there is $g\in G$ such that $Kg\subseteq S$.
\end{enumerate}
\end{definition}

If $G$ is countable discrete then the compact subsets of $G$ are exactly the finite subsets of $G$. If $\pi: G\to G$ is an automorphism, then for any syndetic or thick subset $S$ of $G$, $\pi(S)$ is syndetic or thick, respectively. In particular, if $G$ is a group, $S$ is syndetic or thick, and $h\in G$, then $h^{-1}Sh$ is syndetic or thick, respectively.

The concepts of syndetic and thick are ``dual" in the sense that $S$ is syndetic if and only if for any thick $A$, $S\cap A\neq\varnothing$.

\begin{definition} Let $(X, G)$ be a dynamical system. A point $x\in X$ is {\em uniformly recurrent} if for any open set $U$ with $x\in U$, the set $\{g\in G\,:\, gx\in U\}$ is syndetic.
\end{definition}

The following is a folklore about uniform recurrence.

\begin{proposition}\label{P-minimal vs uniformly recurrent} Let $(X, G)$ be a dynamical system.
\begin{enumerate}
\item If $(X, G)$ is minimal, then every point $x\in X$ is uniformly recurrent.
\item If $x\in X$ is uniformly recurrent, then the orbit closure $\overline{Gx}$ is a minimal $G$-invariant subset of $X$.
\end{enumerate}
\end{proposition}

\begin{theorem}[Auslander \cite{Auslander}, Ellis \cite{Ellis}; also see {\cite[Theorem 8.7]{Fur}}]\label{AE} Let $(X, G)$ be a dynamical system. Then every point of $X$ is proximal to a uniformly recurrent point in its orbit closure.
\end{theorem}

\begin{definition} Let $(X, G)$ be a dynamical system and $x\in X$.
\begin{enumerate}
\item (Following \cite{DLX} and \cite{XY}) $x$ is {\em distal} if it is not proximal to any other point in its orbit closure.
\item $x$ is {\em absolutely distal} if it is not proximal to any other point in $X$.
\end{enumerate}
\end{definition}

In \cite{Fur} (and other early works) absolutely distal points were called distal points. We use the terminology of absolute distality here to distinguish the two notions. It is obvious that absolutely distal points are distal. The converse is false.


If $G$ is a group then the set of all distal points and the set of all absolutely distal points are both $G$-invariant. Theorem \ref{AE} has the following corollary.

\begin{corollary}\label{minimal} A distal point is uniformly recurrent.
\end{corollary}

\subsection{IP*-recurrence}

The concepts of IP-systems, IP-sets, IP*-sets, and IP*-recurrence were all defined by Furstenburg \cite{Fur} for $\mathbb{N}$ or $\mathbb{Z}$. Here we extend them to arbitrary countable semigroups with identity, and introduce a notion of translational IP*-sets for countable groups.

\begin{definition} Let $\mathcal{F}$ be the set of all finite nonempty subsets of $\mathbb{N}$, i.e., $\mathcal{F}=[\mathbb{N}]^{<\omega}\setminus\{\varnothing\}$. An {\em $\mathcal{F}$-sequence} is a sequence $\{x_\alpha\}$ in $G$ indexed by elements $\alpha\in\mathcal{F}$. If $G$ is a semigroup then an {\em IP-system} on $G$ is an $\mathcal{F}$-sequence where
$$ x_{\{i_1,i_2,\dots, i_k\}}=x_{i_1}x_{i_2}\cdots x_{i_k} $$
where $i_1<i_2<\cdots<i_k$. The set $\{x_\alpha\}$ is called an {\em IP-set} on $G$.
\end{definition}

If $\pi: G\to G$ is an automorphism and $\{g_{\alpha}\}$ is an IP-set on $G$, then $\{\pi(g_{\alpha})\}$ is an IP-set. In particular, if $G$ is a group, then for any IP-set $\{g_{\alpha}\}$ on $G$ and $h\in G$, the conjugated set $\{h^{-1}g_{\alpha}h\}$ is an IP-set.

\begin{definition} Let $G$ be a semigroup. A subset $S\subseteq G$ is an {\em IP*-set} if for any IP-set $A$ on $G$, $S\cap A\neq\varnothing$.
\end{definition}
If $G$ contains an identity $e_G$ then any IP*-set must contain $e_G$; this is because, the trivial IP-system with all elements being $e_G$ gives the singleton $\{e_G\}$ as an IP-set.
If $\pi:G\to G$ is an automorphism and $S\subseteq G$ is an IP*-set then so is $\pi(S)$. In particular, if $G$ is a group, $S\subseteq G$ is an IP*-set, and $h\in G$, then $h^{-1}Sh$ is an IP*-set.

\begin{definition} Let $(X, G)$ be a dynamical system and $x\in X$. $x$ is {\em IP*-recurrent} if for any open set $U$ with $x\in U$, the set $\{g\in G\,:\, gx\in U\}$ is an IP*-set.
\end{definition}

\begin{proposition}[\cite{Fur} for $G=\mathbb{N}$ or $\mathbb{Z}$; {\cite[Theorem 4]{DLX}} for the general case]  Let $(X, G)$ be a dynamical system, where $G$ is a semigroup with identity, and $x\in X$. Then $x$ is distal if and only if $x$ is IP*-recurrent.
\end{proposition}

IP*-sets are obviously closed under taking supersets. It follows from the following facts that IP*-sets on any semigroup form a {\em filter}, i.e., the collection consists of nonempty sets and is closed under both intersection and taking supersets.

\begin{proposition}[{\cite[Proposition 8.13]{Fur}}]\label{Ramsey} On any semigroup $G$, IP-sets have the Ramsey property, i.e., if $S\subseteq G$ is an IP-set and $S=S_1\cup S_2$, then either $S_1$ or $S_2$ contains an IP-set.
\end{proposition}

\begin{proposition}[{\cite[Lemma 9.5]{Fur}}] On any semigroup $G$, IP*-sets are closed under intersection.
\end{proposition}

In the rest of this paper we assume that $G$ is a countable discrete group.

\begin{definition} A subset $A$ of $G$ is a {\em translational IP*-set} (or {\em TIP*-set}) if for any $g\in G$, either $Ag$ or $G\setminus Ag$ is an IP*-subset of $G$.
\end{definition}

Since $\{e_G\}$ is an IP-set, every IP*-set contains $e_G$. For a TIP*-set $A$ and $g\in G$, the set $gA$ is an IP*-set if and only if $g^{-1}\in A$. In particular, a TIP*-set $A$ is an IP*-set if and only if $e_G\in A$. Note that for any $g\in G$, $gA$ is an IP*-set if and only if $Ag=g^{-1}(gA)g$ is an IP*-set. It follows that the above definition is equivalent if the right multiplication by $g$ is replaced with the left multiplication by $g$. 

\begin{lemma} TIP*-sets form a Boolean algebra, i.e., they are closed under complementation, union and intersection.
\end{lemma}

\begin{proof} TIP*-sets are obviously closed under complementation by definition. For union and intersection, the closure property follows from the fact that IP*-sets form a filter.
\end{proof}

If $G$ is a countable group, let $(2^G, G)$ denote the Bernoulli right-shift system, i.e., for any $g\in G$ and $x\in 2^G$,
$$ (g\cdot x)(h)=x(hg) $$
for all $h\in G$.

\begin{lemma}\label{lem:sIP*} A point $x\in 2^G$ is distal if and only if $\{g\in G\,:\, x(g)=x(e_G)\}$ is a TIP*-set.
\end{lemma}

\begin{proof} Suppose $x\in 2^G$ is distal. Let $U_0=\{y\in 2^G\,:\, y(e_G)=x(e_G)\}$ and $U_1=\{y\in 2^G\,:\, y(e_G)\neq x(e_G)\}$. Then $U_0$ and $U_1$ are open, and for any $g_0\in G$, $x(g_0)=x(e_G)$ if and only if $g_0x\in U_0$, and $x(g_0)\neq x(e_G)$ if and only if $g_0x\in U_1$. In particular $x\in U_0$. Let $A=\{g\in G\,:\, gx\in U_0\}=\{g\in G\,:\, x(g)=x(e_G)\}$. Since $g_0x$ is also distal, we have that for $x(g_0)=x(e_G)$, the set $\{g\in G\,:\, gg_0x\in U_0\}=Ag_0^{-1}$ is an IP*-set, and for $x(g_0)\neq x(e_G)$, the set $\{g\in G\,:\, gg_0x\in U_1\}=G\setminus Ag_0^{-1}$ is an IP*-set. This shows that $A$ is a TIP*-set.

Conversely, suppose $A=\{g\in G\,:\, x(g)=x(e_G)\}$ is a TIP*-set. Let $e_G\in F$ be a finite subset of $G$ and let $p=x|_F$. Then $U_p=\{y\in 2^G\,:\, \mbox{$y$ extends $p$}\}$ defines a basic open set in $2^G$ with $x\in U_p$. To show that $x$ is distal it is enough to show that $\{g\in G\,:\, gx\in U_p\}$ is an IP*-set. Since
$$ \{g\in G\,:\, gx\in U_p\}=\bigcap_{g_0\in F}\{g\in G\,:\, (gx)(g_0)=p(g_0)=x(g_0)\}, $$
it suffices to show that for any $g_0\in F$, the set
$$ B=\{g\in G\,:\, x(g_0g)=x(g_0)\} $$
is an IP*-set. Now if $x(g_0)=x(e_G)$ then $B=\{g\in G\,:\, x(g_0g)=x(e_G)\}=g_0^{-1}A$ is an IP*-set since it contains $e_G$; if $x(g_0)\neq x(e_G)$ then $B=G\setminus g_0^{-1}A$ is an IP*-set for the same reason.
\end{proof}

\section{Denseness of Almost Automorphic Points\label{sec:3}}

In this section we study the notion of almost automorphic points, which is closely related to but (strictly) stronger than distality. We will give a complete characterization of those countable groups $G$ for which the set of almost automorphic points in $(2^G, G)$ is dense. 

\subsection{Characterizations of almost automorphic points}

Furstenburg \cite{Fur} also introduced the concepts of $\Delta$-sets, $\Delta^*$-sets and $\Delta^*$-recurrence in connection to the concept of almost automorphic points. Here we extend them to arbitrary countable groups and introduce the notion of T$\Delta^*$-sets.

We continue to assume that $G$ is a countable discrete group.

\begin{definition} Let $A, S\subseteq G$.
\begin{enumerate}
\item $A$ is a {\em $\Delta$-set} if there is a sequence $\{g_n\}$ in $G$ such that $A=\{g_n^{-1}g_m\,:\, n<m\}$.
\item $S$ is a {\em $\Delta^*$-set} if for any $\Delta$-set $A$, $S\cap A\neq\varnothing$.
\item $S$ is a {\em translational $\Delta^*$-set} (or {\em T$\Delta^*$-set}) if for any $g\in G$, either $Sg$ or $G\setminus Sg$ is a $\Delta^*$-set.
\end{enumerate}
\end{definition}

\begin{lemma} Every IP-set contains a $\Delta$-set. Consequently, every $\Delta^*$-set is an IP*-set.
\end{lemma}

\begin{proof} Let $S=\{g_{i_1}\cdots g_{i_k}\,:\, i_1<\cdots<i_k\}$ be an IP-set. Let $h_n=g_1\cdots g_n$. Then for $n<m$, $h_n^{-1}h_m\in S$. So $S$ is a $\Delta$-set.
\end{proof}

Note that every $\Delta^*$-set contains $e_G$.

\begin{lemma}\label{lem:R} $\Delta$-sets have the Ramsey property, i.e., if $A_1\cup A_2$ is a $\Delta$-set, then either $A_1$ or $A_2$ contains a $\Delta$-set.
\end{lemma}

\begin{proof} Suppose $\{g_n\}$ is a sequence in $G$ such that $A_1\cup A_2=\{g_n^{-1}g_m\,:\, n<m\}$. Let $P_1=\{(n,m)\,:\, n<m \mbox{ and } g_n^{-1}g_m\in A_1\}$ and $P_2=\{(n,m)\,:\, n<m \mbox{ and } g_n^{-1}g_m\in A_2\}$. By Ramsey's theorem \cite[Theorem A]{Ra}, there is an infinite sebsequence $\{k_n\}$ such that either $\{(k_n,k_m)\,:\, n<m\}\subseteq P_1$ or $\{(k_n,k_m)\,:\, n<m\}\subseteq P_2$. In the first case, $P_1$ contains the set $\{g_{k_n}^{-1}g_{k_m}\,:\, n<m\}$, and in the second case, $P_2$ contains this set.
\end{proof}

\begin{lemma} If $S$ is a $\Delta^*$-set and $A$ is a $\Delta$-set, then $S\cap A$ contains a $\Delta$-set.
\end{lemma}

\begin{proof} Since $A=(A\cap S)\cup (A\setminus S)$, it follows from the above lemma that either $A\cap S$ or $A\setminus S$ contains a $\Delta$-set. Since $(A\setminus S)\cap S=\varnothing$, $A\setminus S$ does not contain a $\Delta$-set. Thus $A\cap S$ must contain a $\Delta$-set.
\end{proof}


\begin{lemma} T$\Delta^*$-sets form a Boolean algebra.
\end{lemma}
\begin{proof} It is clear from the definition that T$\Delta^*$-sets are closed under complementation. By definition and the previous lemma, $\Delta^*$-sets are closed under intersection and taking supersets. It follows that T$\Delta^*$-sets are closed under union and intersection.
\end{proof}

\begin{definition} Let $(X, G)$ be a dynamical system and $x\in X$.
\begin{enumerate}
\item $x$ is {\em $\Delta^*$-recurrent} if for any open set $U$ with $x\in U$, the set $\{g\in G\,:\, gx\in U\}$ is a $\Delta^*$-set.
\item $x$ is an {\em almost automorphic point} if for any sequence $\{g_n\}$ in $G$, if $g_nx\to y$ then $g_n^{-1}y\to x$.
\end{enumerate}
\end{definition}

\begin{lemma}\label{lem:3.7} Let $(X, G)$ be a dynamical system and $x\in X$. Then $x$ is $\Delta^*$-recurrent
if and only if $x$ is an almost automorphic point.
\end{lemma}

\begin{proof} First suppose $x$ is $\Delta^*$-recurrent. Suppose $g_nx\to y$. By passing down to a subsequence if necessary, suppose $g_n^{-1}y\to z$. Assume $x\neq z$. Let $U_x$ and $U_z$ be disjoint open subsets of $X$ with $x\in U_x$ and $z\in U_z$. Since the set $\{g\in G\,:\, gx\in U_x\}$ is a $\Delta^*$-set, its intersection with the $\Delta$-set $\{g_n^{-1}g_m\,:\, n<m\}$ must contain a $\Delta$-set. In fact, from the proof of Lemma~\ref{lem:R} we see that there is an infinite subsequence $\{k_n\}$ such that the intersection contains the $\Delta$-set $\{g_{k_n}^{-1}g_{k_m}\,:\, n<m\}$. Note that $g_{k_n}^{-1}y\to z$. Thus there exist an $n$ and an open set $U_y$ with $y\in U_y$ such that $g_{k_n}^{-1}U_y\subseteq U_z$. Since $g_k x\to y$, there is an $m>n$ such that $g_{k_m}x\in U_y$. Hence $g_{k_n}^{-1}g_{k_m}x\in U_z$. This contradicts $g_{k_n}^{-1}g_{k_m}x\in U_x$.

Conversely, suppose $x$ is almost automorphic. To show that $x$ is $\Delta^*$-recurrent, we consider an open set $U$ with $x\in U$ and a $\Delta$-set $\{g_n^{-1}g_m\,:\, n<m\}$ in $G$. We need to show that there exist $n<m$ such that $g_n^{-1}g_mx\in U$. Pass to a subsequence $\{g_{n_k}\}$ such that $g_{n_k}x\to y$ for some $y\in X$. Then $g_{n_k}^{-1}y\to x$. Choose $j$ and an open subset $U_y$ with $y\in U_y$ such that $g_{n_j}^{-1}U_y\subseteq U$. Let $k>j$ be such that $g_{n_k}x\in U_y$. Then $g_{n_j}^{-1}g_{n_k}x\in U$.
\end{proof}

It follows that every almost automorphic point is distal. In particular, if the set of all almost automorphic points in $(2^G, G)$ is dense, then so is the set of all distal points.

For $\Delta^*$-recurrence, we also have a lemma similar to Lemma \ref{lem:sIP*}.

\begin{lemma}\label{lem:3.8} $x\in 2^G$ is $\Delta^*$-recurrent if and only if $\{g\in G\,:\, x(g)=x(e_G)\}$ is a T$\Delta^*$-set.
\end{lemma}

\begin{proof} The proof is similar to that of Lemma~\ref{lem:sIP*}.
\end{proof}

\subsection{Denseness of almost automorphic points}
 
We give an answer of Question \ref{AAP} in this subsection. It turns out that there is a simple characterization of all countable discrete groups for which this question has a positive answer. In this section we will prove that for a countably infinite discrete group $G$, the set of all almost automorphic points is dense in $(2^G, G)$ if and only if $G$ is maximally almost periodic. 

\begin{definition}[{Following \cite{vN}}] A topological group is {\em maximally almost periodic} (or {\em MAP}) if the continuous homomorphisms from $G$ into compact Hausdorff groups separate the points of $G$.
\end{definition}

By definition, an MAP group embeds algebraically into a compact Hausdorff group. In our context, since $G$ is countable discrete, by \cite[Section 2.20]{MZ}, $G$ is MAP if and only if it embeds algebraically into a compact Polish group. It is well known that countable abelian groups are MAP (\cite[Theorem 36]{vN}). If $G$ is a countable residually finite group then $G$ is embedded into its profinite completion (see e.g. \cite[Section 3.2]{RZbook}), hence it is MAP. 



\begin{lemma}\label{lem:tbli} Let $G$ be a countable discrete group. If $G$ admits a totally bounded, left-invariant metric $d$ generating a topology $\tau_d$ on $G$, then any $\tau_d$-clopen nbhd of $e_G$ is a T$\Delta^*$-set.
\end{lemma}

\begin{proof} It suffices to show that any $\tau_d$-open nbhd of $e_G$ is a $\Delta^*$-set. Let $A$ be a $\tau_d$-open set with $e_G\in A$. Let $\epsilon>0$ be such that $\{g\,:\, d(g, e_G)<\epsilon\}\subseteq A$. Let $B$ be a finite $\epsilon/2$-net, i.e., for any $g\in G$ there is $b\in B$ such that $d(g, b)<\epsilon/2$. Let $\{g_n\}$ be any sequence in $G$. There are $0\leq i<j\leq |B|$ such that for some $b\in B$, $d(g_i, b), d(g_j, b)<\epsilon/2$. Then $d(g_i, g_j)<\epsilon$. By the left-invariance of $d$, $d(g_i^{-1}g_j, e_G)<\epsilon$. Thus $g_i^{-1}g_j\in A$.
\end{proof}

Recall the notion of an enveloping semigroup of a dynamical system $(X, G)$. Equipp $X^X$ with the product topology and consider the natural map $\iota\colon G\to X^X$ given by $\iota(g)(x)=gx$. The closure of $\iota(G)$ in $X^X$ is denoted $E(X)$ and called the {\em enveloping semigroup} (or {\em Ellis semigroup}) of $(X, G)$. 

\begin{theorem} Let $G$ be a countable discrete group. Then the following are equivalent:
\begin{enumerate}
\item[(1)] $G$ is MAP;
\item[(2)] $G$ admits a totally bounded, left-invariant metric;
\item[(3)] The set of all almost automorphic points in $(2^G, G)$ is dense. 
\end{enumerate}
\end{theorem}

\begin{proof} (1)$\Rightarrow$(2): Every compact Polish group admits a totally bounded, two-sided invariant compatible metric. If a countable group $G$ is MAP, then naturally it admits a totally bounded, left-invariant metric.

(2)$\Rightarrow$(3): Let $d$ be a totally bounded, left-invariant metric which generates a topology $\tau_d$ on $G$. Let $g_1,\dots, g_n, h_1,\dots, h_m\in G$ be distinct nonidentity elements. By Lemmas~\ref{lem:3.7} and \ref{lem:3.8}, it suffices to define a T$\Delta^*$-set $S$ with $g_1,\dots, g_n\in S$ and $h_1,\dots, h_m\not\in S$.

Let $Q=\{d(g,h)\,:\, g,h\in G\}$. Since $G$ is countable, $Q$ is countable. Let
$$r=\min\{d(e_G, h_k), d(g_i, h_k)\,:\, 1\leq i\leq n, 1\leq k\leq m\}.$$
Let $\epsilon>0$ be such that $\epsilon< r$ and $\epsilon\not\in Q$. Then for any $g\in G$, $\{h\in G\,:\, d(g,h)<\epsilon\}=\{h\in G\,:\, d(g, h)\leq \epsilon\}$ is a $\tau_d$-clopen set. Define
$$ S=\{h\in G\,:\, d(h, e_G)<\epsilon\}\cup\bigcup_{1\leq i\leq n} \{h\in G\,:\, d(h, g_i)<\epsilon\}. $$
Then $S$ is $\tau_d$-clopen, and $e_G, g_1, \dots, g_n\in S$ but $h_1,\dots, h_m\not\in S$. By Lemma~\ref{lem:tbli}, $S$ is a T$\Delta^*$-set.

(3)$\Rightarrow$(1): For every $g\in G$ such that $g\ne e_G$, by the denseness of the set of all almost automorphic points in $(2^G, G)$, there is an almost automorphic point $x\in 2^G$ such that $x(e_G)=0$ and $x(g)=1$. By Corollary \ref{minimal}, $(\overline{G x},G)$ is a minimal system and $x$ is an almost automorphic point in this system. Then by  \cite[Section 3.4]{Veech1}, there is an equicontinuous system $(Y,G)$ and  a factor map $f:\overline{G x}\to Y$ such that for every $y\in X$, $f^{-1}(f(y))=\{y\}$ if and only if $y$ is almost automorphic. So we have that $f(x)\ne f(gx)=gf(x)$. Let $\varphi:G\to{\rm Homeo}(Y)$ denote the action of $G$ on $Y$, then $\varphi(g)\ne {\rm id}(Y)$. The enveloping semigroup $E(Y)$ of $(Y,G)$ is a compact group (see \cite[Chapter 3, Theorem 3]{AusBook}) and $\varphi(G)$ is dense in $E(Y)$. In particular, $\varphi$ is a group homomorphism from $G$ to some compact Hausdorff topological group such that $\varphi(g)$ is not the identity element.

For every $g\in G$ such that $g\ne e_G$, let $H_g$ be a compact Hausdorff topological group and $\varphi_g:G\to H_g$ be a group homomorphism such that $\varphi_g(g)\ne e_{H_g}$. Then $\prod_{g\ne e_G}\varphi_g$ is an embedding of $G$ to some compact Hausdorff topological group, so $G$ is MAP.
\end{proof}

\section{Denseness of Distal Points\label{sec:4}}

In this section we will give some characterizations of countable discrete groups $G$ for which the set of distal points in $(2^G, G)$ is dense. We will also give two examples of countable 2-step nilpotent groups $G_1$ and $G_2$, neither of which is MAP, such that the set of distal points in $(2^{G_1},G_1)$ is dense while the set of distal points in $(2^{G_2},G_2)$ is not dense. Recall that countable abelian groups are all MAP.  Thus our results suggest that there is unlikely any characterization of groups for which the distal points are dense in terms of their algebraic properties. 

Throughout this section we let $G$ be a countable discrete group with identity element $e_G$. We consider continuous actions of $G$ on compact metrizable spaces. Let $G$ act on a compact metrizable space $X$ continuously. The action is said to be {\em effective} if for each $g\in G\setminus \{e_G\}$ there  is some $x\in X$ with $gx\neq x$, i.e., the corresponding group homomorphism from $G$ to the homeomorphism group of $X$ is injective. 
We say a point $x\in X$ has {\em trivial stabilizer} if $gx\neq x$ for all $g\in G\setminus \{e_G\}$.
We say the action is {\em point-distal} if there is a distal point in $X$ with dense orbit. 

The following is a well-known construction of zero-dimensional extensions. For a finite set $A$, $z\in A^G$ and $g\in G$, denote $z(g)$ by $z_g$ for notational simplicity. 

\begin{lemma} \label{L-zero-dim extension}
Let $G$ act on a compact metrizable space $X$ continuously. Let $\rho$ be a compatible metric on $X$. Let $\{\cU_n\}_{n\in \Nb}$ be a sequence of finite families of nonempty open subsets of $X$ satisfying the following conditions:
\begin{enumerate}
\item[(a)] for each $n\in \Nb$, the members of $\cU_n$ are pairwise disjoint and the union $\bigcup \cU_n$ is dense in $X$;
\item[(b)] for each $n\in \Nb$, every member of $\cU_{n+1}$ is contained in some member of $\cU_n$;
\item[(c)] ${\rm mesh}(\cU_n):=\max_{U\in \cU_n}\diam(U)$ converges to $0$ as $n\to \infty$.
\end{enumerate}
For each $n\in \Nb$, denote by $Z_n$ the set of $z\in (\cU_n)^G$ satisfying that $\bigcap_{g\in G} g^{-1}\overline{z_g}\neq \varnothing$. Then $Z_n$ is a closed $G$-invariant subset of $(\cU_n)^G$.  The map $\pi_n: Z_{n+1}\rightarrow Z_n$ defined by $\pi_n(z)_g\supseteq z_g$ for all $z\in Z_{n+1}$ and $g\in G$ is a factor map. The map $\pi: Z=\varprojlim Z_n\rightarrow X$ sending $(z_n)_{n\in \Nb}$ to the point in $\bigcap_{n\in \Nb} \overline{z_{n, e_G}}$ is a factor map. 
\end{lemma}
\begin{proof} Let $n\in \Nb$. Note that $z\in (\cU_n)^G$  is in $Z_n$ if and only if for every nonempty finite subset $F$ of $G$ one has  $\bigcap_{g\in F} g^{-1}\overline{z_g}\neq \varnothing$. It follows that $Z_n$ is closed in $(\cU_n)^G$. For any $z\in Z_n$, one has
$$ \bigcap_{g\in G} g^{-1}\overline{(hz)_g}= \bigcap_{g\in G} g^{-1}\overline{z_{gh}}=h\bigcap_{g\in G} (gh)^{-1}\overline{z_{gh}}\neq \varnothing$$
for each $h\in G$, whence $hz\in Z_n$. Thus $Z_n$ is $G$-invariant. 

Let $z\in Z_{n+1}$. Define $\pi_n(z)\in (\cU_n)^G$ by $\pi_n(z)_g\supseteq z_g$ for all $g\in G$. The element $\pi_n(z)$ is well defined because of conditions (a) and (b). Then 
$$ \bigcap_{g\in G} g^{-1}\overline{\pi_n(z)_g}\supseteq \bigcap_{g\in G}g^{-1}\overline{z_g}\neq \varnothing,$$
whence $\pi_n(z)\in Z_n$. Clearly $\pi_n$ is continuous and $G$-equivariant. 

From condition (a) we see that $X$ is the union of $\overline{V}$ for $V\in \cU_{n+1}$.  Furthermore, for any $U\in \cU_n$, $U$ is disjoint with $\overline{V}$ if $V\in \cU_{n+1}$ and $V\subseteq U'$ for some $U'\in \cU_n\setminus \{U\}$. Thus, combined with condition (b), we see that every $U\in \cU_n$ is covered by $\overline{V}$ for $V\in \cU_{n+1}$ satisfying $V\subseteq U$. It follows that, for every $U\in \cU_n$, its closure $\overline{U}$ is the union of $\overline{V}$ for $V\in \cU_{n+1}$ satisfying $V\subseteq U$. 

Let $y\in Z_n$. Take an $x\in \bigcap_{g\in G} g^{-1}\overline{y_g}$. Then $gx\in \overline{y_g}$ for every $g\in G$. For each $g\in G$, from the above paragraph we can find some $z_g\in \cU_{n+1}$ with $z_g\subseteq y_g$ such that $gx\in \overline{z_g}$. Then $z=(z_g)_{g\in G}$ is in $Z_{n+1}$ and $\pi_n(z)=y$. This shows that $\pi_n$ is surjective. Therefore $\pi_n$ is a factor map. 

Set $Z=\varprojlim_{n\to \infty} Z_n$. Let $z=(z_n)_{n\in \Nb}\in Z$. Then $z_{n+1, e_G}\subseteq z_{n, e_G}$ for all $n\in \Nb$, thus $\bigcap_{n\in \Nb}\overline{z_{n, e_G}}$ is nonempty. From condition (c) we see that $\bigcap_{n\in \Nb}\overline{z_{n, e_G}}$ is a singleton. 
Define $\pi(z)$ to be the point in $\bigcap_{n\in \Nb}\overline{z_{n, e_G}}$. From condition (c) we see that $\pi: Z\rightarrow X$ is continuous. 
Note that we also have $\bigcap_{g\in G} g^{-1}\overline{z_{n+1, g}}\subseteq \bigcap_{g\in G} g^{-1}\overline{z_{n, g}}$ for all $n\in \Nb$, whence $\bigcap_{n\in \Nb}\bigcap_{g\in G} g^{-1}\overline{z_{n, g}}\neq \varnothing$. Since
$$\{\pi(z)\}=\bigcap_{n\in \Nb}\overline{z_{n, e_G}}\supseteq \bigcap_{n\in \Nb}\bigcap_{g\in G} g^{-1}\overline{z_{n, g}}\neq \varnothing,$$
it follows that
\begin{align} \label{E-zero-dim extension}
\{\pi(z)\}= \bigcap_{n\in \Nb}\bigcap_{g\in G} g^{-1}\overline{z_{n, g}}.
\end{align}
For any $h\in G$, we have
\begin{align*}
\{h\pi(z)\}&= h\bigcap_{n\in \Nb}\bigcap_{g\in G} g^{-1}\overline{z_{n, g}}=\bigcap_{n\in \Nb}\bigcap_{g\in G} (gh^{-1})^{-1}\overline{z_{n, g}}
=\bigcap_{n\in \Nb}\bigcap_{g\in G} g^{-1}\overline{z_{n, gh}}\\
&=\bigcap_{n\in \Nb}\bigcap_{g\in G} g^{-1}\overline{(hz)_{n, g}}=\{\pi(hz)\},
 \end{align*}
 whence $h\pi(z)=\pi(hz)$. Thus $\pi$ is equivariant. 
 
 Let $x\in X$. For each $g\in G$, by condition (a) we can find some $z_{1, g}\in \cU_1$ such that $gx\in \overline{z_{1, g}}$. Then $z_1:=(z_{1, g})_{g\in G}$ is in $Z_1$. For each $g\in G$, from the third paragraph of the proof we can also find some $z_{2, g}\in \cU_2$ such that $z_{2, g}\subseteq z_{1, g}$ and $gx\in \overline{z_{2, g}}$. Then $z_2:=(z_{2, g})_{g\in G}$ is in $Z_2$ and $\pi_1(z_2)=z_1$. In this way, we can find $z_{n, g}\in \cU_n$ for each $n\in \Nb$ and $g\in G$ such that $gx\in \overline{z_{n, g}}$ for all $n\in \Nb$ and $g\in G$, $z_n:=(z_{n, g})_{g\in G}$ is in $Z_n$ for all $n\in \Nb$, and $\pi_n(z_{n+1})=z_n$ for all $n\in \Nb$. Then $z=(z_n)_{n\in \Nb}$ is in $Z$, and $\pi(z)=x$. This shows that $\pi$ is surjective. Thus $\pi$ is a factor map. 
\end{proof}

\begin{lemma} \label{L-unique preimage}
Let $G$ act on a compact metrizable space $X$ continuously. Let $\{\cU_n\}_{n\in \Nb}$ be a sequence of finite families of nonempty open subsets of $X$ satisfying the conditions in Lemma~\ref{L-zero-dim extension}. Let $Z_n$, $Z$ and $\pi$ be given in Lemma~\ref{L-zero-dim extension}. Let $x\in X$ such that the orbit of $x$ is contained in $\bigcup \cU_n$ for every $n\in \Nb$. Then $|\pi^{-1}(x)|=1$.
\end{lemma}
\begin{proof} Let $z=(z_n)_{n\in \Nb}$ with $\pi(z)=x$. From \eqref{E-zero-dim extension} we have $gx\in \overline{z_{n, g}}$ for all $n\in \Nb$ and $g\in G$. Since the orbit of $x$ is contained in $\bigcup \cU_n$ for each $n\in \Nb$, from condition (a) of Lemma~\ref{L-zero-dim extension} we have $gx\in z_{n, g}$ for all $n\in \Nb$ and $g\in G$. Applying condition (a) of Lemma~\ref{L-zero-dim extension} again,  we see that $z_{n, g}$ is unique for every $n\in \Nb$ and $g\in G$. Thus $z$ is unique.  
\end{proof}

\begin{lemma} \label{L-avoid orbit}
Let $G$ act on a compact metrizable space $X$ continuously. Let $x\in X$ with dense orbit. Let $\rho$ be a compatible metric on $X$. Then there is a sequence $\{\cU_n\}_{n\in \Nb}$ of finite families of nonempty open subsets of $X$ such that the following hold:
\begin{enumerate}
\item[(a)] for each $n\in \Nb$, the members of $\cU_n$ are pairwise disjoint and the union $\bigcup \cU_n$ is dense in $X$;
\item[(b)] for each $n\in \Nb$, every member of $\cU_{n+1}$ is contained in some member of $\cU_n$;
\item[(c)] ${\rm mesh}(\cU_n):=\max_{U\in \cU_n}\diam(U)$ converges to $0$ as $n\to \infty$;
\item[(d)] for each $n\in \Nb$, the orbit of $x$ is contained in $\bigcup \cU_n$.
\end{enumerate}
\end{lemma}
\begin{proof} The lemma is trivial when $X$ is finite. Thus we may assume that $X$ is infinite. We shall construct $\{\cU_n\}_{n\in \Nb}$ inductively on $n$ satisfying conditions (a), (b), (d) and the following condition which is stronger than condition (c):
\begin{enumerate}
\item[(c')] for each $n\in \Nb$,  ${\rm mesh}(\cU_n)\le \diam(X)/2^{n-1}$. 
\end{enumerate}

To start, we take $\cU_1=\{X\}$. Assume that we have constructed $\cU_1, \dots, \cU_n$ for some $n\in \Nb$. Since $X$ is compact, for each $U\in \cU_n$ we can find a finite subset $A_U$ of $U$ such that for every $x\in U$, there is some $a\in A_U$ with $\rho(x, a)<r_n:=\diam(X)/2^{n+1}$. Set $A:=\bigcup_{U\in \cU_n}A_U$. Since $X$ is infinite and the orbit of $x$ is dense in $X$, the orbit of $x$ is infinite. Enumerate the points in the orbit of $x$ as $x_1, x_2, \dots$. 

For any $y\in X$ and $r>0$, we write $B(y, r)$ for the open ball $\{z\in X: \rho(z, y)<r\}$. 
We shall construct open sets $V_{a, j}$ for  $a\in A$ and $j\in \Nb$ such that the following hold:
\begin{enumerate}
\item[(i)] for each $U\in \cU_n$ and $a\in A_U$, one has $V_{a, j}\subseteq V_{a, j+1}$ and $\overline{V_{a, j}}\subseteq U\cap B(a, r_n)$ for all $j\in \Nb$;
\item[(ii)] for each $j\in \Nb$, the sets $\overline{V_{a, j}}$ for $a\in A$ are pairwise disjoint, and $x_j\in \bigcup_{a\in A} V_{a, j}$. 
\end{enumerate}
Assume that we have constructed such $V_{a, j}$. Set $V_a=\bigcup_{j\in \Nb}V_{a, j}$ for $a\in A$, and $\cU_{n+1}=\{V_a: a\in A, V_a\neq \varnothing\}$. Then the members of $\cU_{n+1}$ are pairwise disjoint. Since the orbit of $x$ is dense in $X$, and the orbit of $x$ is contained in $\bigcup \cU_{n+1}$, we see that $\bigcup \cU_{n+1}$ is dense in $X$. For each $U\in \cU_n$ and $a\in A_U$, one has $V_a\subseteq U\cap B(a, r_n)$. Note that $\mesh(\cU_{n+1})\le 2\cdot r_n=\diam(X)/2^n$. Thus $\cU_1, \dots, \cU_{n+1}$ satisfy conditions (a), (b), (c') and (d).

Now it suffices to construct open sets $V_{a, j}$ satisfying (i) and (ii). We construct them inductively on $j$. 

Let $j=1$. By condition (d), the point $x_1$ is in some $U\in \cU_n$. Then we can find some $a\in A_U$ such that $x_1\in B(a, r_n)$. Take a small $r>0$ with $\overline{B(x_1, r)}\subseteq U\cap B(a, r_n)$. Setting $V_{a, 1}=B(x_1, r)$ and $V_{b, 1}=\varnothing$ for all $b\in A\setminus \{a\}$ satisfies the requirement. 

Assume that we have constructed $V_{a, j}$ for some $j\in \Nb$ and all $a\in A$. 
We separate the discussion into 3 cases. 

Case I: $x_{j+1}\in \bigcup_{b\in A} V_{b, j}$. In this case we may set $V_{b, j+1}=V_{b, j}$ for all $b\in A$. 

Case II: $x_{j+1}\not\in \bigcup_{b\in A} \overline{V_{b, j}}$. By condition (d), the point $x_{j+1}$ is in some $U\in \cU_n$. Then we can find some $a\in A_U$ such that $x_{j+1}\in B(a, r_n)$. Take a small $r>0$ such that  $\overline{B(x_{j+1}, r)}\subseteq U\cap B(a, r_n)$ and 
$\overline{B(x_{j+1}, r)}$ is disjoint with $\bigcup_{b\in A} \overline{V_{b, j}}$. Setting $V_{a, j+1}=B(x_{j+1}, r)\cup V_{a, j}$ and $V_{b, j+1}=V_{b, j}$ for all $b\in A\setminus \{a\}$ satisfies the requirement. 

Case III: $x_{j+1}\in \partial V_{a, j}$ for some $a\in A$. Then $a$ is in some $U\in \cU_n$. By the condition (ii) we have $x_{j+1}\not\in \overline{V_{b, j}}$ for all $b\in A\setminus \{a\}$ and $x_{j+1}\in U\cap B(a, r_n)$. Take a small $r>0$ such that  $\overline{B(x_{j+1}, r)}\subseteq U\cap B(a, r_n)$ and 
$\overline{B(x_{j+1}, r)}$ is disjoint with $\bigcup_{b\in A\setminus \{a\}} \overline{V_{b, j}}$. Setting $V_{a, j+1}=B(x_{j+1}, r)\cup V_{a, j}$ and $V_{b, j+1}=V_{b, j}$ for all $b\in A\setminus \{a\}$ satisfies the requirement. 

This finishes the construction of $V_{a, j+1}$ for $a\in A$.
\end{proof}

The next lemma shows that the extension in Lemma \ref{L-zero-dim extension} (along the inverse direction of the factor map) can somehow preserve distal points.

\begin{lemma} \label{L-one-one to point-distal}
Let $G$ act on compact metrizable spaces $Z$ and $X$ continuously and let $\pi: Z\rightarrow X$ be a factor map. 
Let $z$ be a point in $Z$ such that $\pi^{-1}(\pi(z))=\{z\}$ and $\pi(z)$ is a distal point in $X$.  Then $z$ is a distal point in $Z$. 
\end{lemma}
\begin{proof} Let $y\in Z$ be in the orbit closure $\overline{Gz}$ of $z$ such that $(z, y)$ is proximal. Then $\pi(y)$ is in the orbit closure of $\pi(z)$ and $(\pi(z), \pi(y))$ is proximal. Since $\pi(z)$ is distal, we have $\pi(y)=\pi(z)$. As $\pi^{-1}(\pi(z))=\{z\}$, we conclude that $y=z$. Thus $z$ is distal.
\end{proof}

We are now ready to prove the main theorem of this section.

\begin{theorem} \label{P-effective point-distal}
Let $G$ be  a countable discrete group and let $A$ be a finite set with $|A|\ge 2$. The following are equivalent:
\begin{enumerate}
\item The set of distal points in $(A^G,G)$ is dense in $A^G$;
\item $G$ has an effective point-distal continuous action on some compact metrizable space;
\item For each $g\in G\setminus \{e_G\}$, there is a point-distal continuous action of $G$ on a compact metrizable space $X$ such that 
           $gx\neq x$ for some $x\in X$;
\item For each $g\in G\setminus \{e_G\}$, there is  a point-distal continuous  action of $G$ on a compact metrizable space $X$ such that $gx\neq x$ for some distal point $x\in X$;
\item    There exist a minimal continuous action of $G$ on a compact metrizable space $X$ and  a distal point in $X$ with trivial stabilizer;
\item    There exist a minimal continuous action of $G$ on a zero-dimensional compact metrizable space $X$ and  a distal point in $X$ with trivial stabilizer.
\end{enumerate}
\end{theorem}
\begin{proof} (1)$\Rightarrow$(3): Let $g\in G\setminus \{e_G\}$. Since $|A|\ge 2$, by (1) we can find some distal point $x\in A^G$ such
$x_{e_G}\neq x_g$. Then the $G$-action on $\overline{Gx}$ is minimal by Corollary \ref{minimal}. Since 
$$(gx)_{e_G}=x_g\neq x_{e_G},$$
we have $gx\neq x$. 

(3)$\Rightarrow$(4): Let $g\in G\setminus \{e_G\}$. Let $X$ and $x\in X$ witness (3) for $g$. By a result of Ellis \cite{Ellis73} (see also \cite[VI.6.4.6]{deVries}) there is a dense $G_\delta$ subset of $X$ consisting of distal points. Thus, perturbing $x$ if necessary, we may assume that $x$ is distal. 

(4)$\Rightarrow$(5): When $G=\{e_G\}$, (5) holds trivially. Thus we may assume that $G\neq \{e_G\}$. For each $g\in G\setminus \{e_G\}$, by (4) we can find a point-distal continuous action of $G$ on a compact metrizable space $X_g$ and a distal point $x_g\in X_g$ such that $gx_g\neq x_g$. Consider the diagonal action of $G$ on $\prod_{g\in G\setminus \{e_G\}}X_g$. Then $x:=(x_g)_{g\in G\setminus \{e_G\}}\in \prod_{g\in G\setminus \{e_G\}}X_g$ is distal and has trivial stabilizer.
Set $X$ to be the orbit closure of $x$ in $\prod_{s\in G\setminus \{e_G\}}X_s$. Then the $G$-action on $X$ is minimal.

(5)$\Rightarrow$(6): Let $X$ and $x\in X$ witness (5). Take a compatible metric $\rho$ on $X$. By Lemma~\ref{L-avoid orbit}  we can find a sequence $\{\cU_n\}_{n\in \Nb}$ of finite families of nonempty open subsets of $X$ satisfying the conditions in Lemma~\ref{L-avoid orbit}. Then we have the subshift $Z_n$ for each $n\in \Nb$ and the factor map $\pi: Z=\varprojlim Z_n\rightarrow X$ given by Lemma~\ref{L-zero-dim extension}. Since each $Z_n$ is zero-dimensional, so is $Z$. By Lemma~\ref{L-unique preimage} and condition (d) of Lemma~\ref{L-avoid orbit} we have $|\pi^{-1}(x)|=1$. Say, $\pi(z)=x$. By Lemma~\ref{L-one-one to point-distal} the point $z$ is distal in $Z$. Since $x$ has trivial stabilizer, so does $z$. Denote by $Y$ the orbit closure of $z$ in $Z$. Then $Y$ is zero-dimensional, and $z$ is distal in $Y$. Thus by Corollary~\ref{minimal} and Proposition~\ref{P-minimal vs uniformly recurrent} the $G$-action on $Y$ is minimal. 

(6)$\Rightarrow$(1): Let $X$ and $x\in X$ witness (6). Let $U$ be a nonempty open subset of $A^G$. Then there are a nonempty finite subset $F$ of $G$ and a $p\in A^F$ such that the clopen set 
$$A_p:=\{y\in A^G: y|_F=p\}$$
is contained in $U$. Since $x$ has trivial stabilizer, the points $gx$ for $g\in F$ are distinct. As $X$ is zero-dimensional, we can find a clopen partition $X=\bigsqcup_{a\in A} X_a$ of $X$ such that for each $g\in F$ one has $gx\in X_{p(g)}$. Define a map $\pi: X\rightarrow A^G$ by
$$\pi(z)_g=a \mbox{ if } gz\in X_a, a\in A,$$ 
for $z\in X$ and $g\in G$. Then $\pi$ is continuous and $G$-equivariant. Thus $\pi$ is a factor map from $X$ to $\pi(X)$. For any factor map $X'\rightarrow Y'$ between minimal continuous actions of $G$ on compact metrizable spaces, the images of distal points are distal (see \cite[Corollary IV.3.25.2]{deVries}). Thus $\pi(x)$ is distal in $\pi(X)$. Note that $\pi(x)\in A_p\subseteq U$. Therefore (1) holds. 

(5)$\Rightarrow$(2)$\Rightarrow$(3) are trivial. 
\end{proof}

Next we present an example of a $2$-step nilpotent group $G$ which is not MAP but the set of distal points in $2^G$ is dense. Together with Example~\ref{G_2} we will have two examples of 2-step nilpotent groups with different behaviors regarding the denseness of their distal points. 

\begin{example}\label{E-effective point-distal action}
Let $p$ be a prime integer. Denote by $G$ the set of nonzero rational numbers which can be written as $\frac{a}{b}$ such that $a, b$ are integers not divisible by $p$. Then $G$ is a group under multiplication. The group $G$ acts on the additive group $\Qb$ by group automorphisms via multiplication. Denote by $G_1$ the semidirect product $\mathbb{Q}\rtimes G$. It is easily checked that $G_1$ is 2-step nilpotent. It is shown in \cite[Example 4.6]{Abels83} that $G_1$ has an effective distal continuous action on a compact metrizable space. In particular, $G_1$ satisfies the conditions in Theorem~\ref{P-effective point-distal}. On the other hand, it is also shown in \cite[Example 4.6]{Abels83} that $G_1$ is not MAP. 
\end{example}

\section{Point-Distal Radicals\label{sec:6}}

In this section we define a notion of point-distal radical of a countable discrete group $G$ and show that the denseness of distal points in $(2^G, G)$ implies that the point-distal radical of $G$ is trivial. 

Let $H$ be a normal subgroup of $G$ and let $\Gamma$ be a compact  metrizable group. The space $\Hom(H, \Gamma)$ of group homomorphisms $H\rightarrow \Gamma$ is a closed subset of $\Gamma^H$, thus is a  compact metrizable space. The group $G$ acts on $\Hom(H, \Gamma)$ continuously via conjugation: 
$$g\phi(h)=\phi(g^{-1}hg)$$
for $\phi\in \Hom(H, \Gamma)$, $g\in G$, and $h\in H$. 

\begin{definition} We say $\phi\in \Hom(H, \Gamma)$ is {\it $G$-point-distal} if $\phi$ is distal in its orbit closure under the $G$-action. 
\end{definition}

\begin{remark} \label{R-extension to distal}
Let $G$ be a countable discrete group and let $H$ be a normal subgroup of $G$. If $\Gamma$ is a compact metrizable group and $\phi\in \Hom(H, \Gamma)$ extends to a group homomorphism $G\rightarrow \Gamma$, then $\phi$ is $G$-point-distal. Indeed, $\Gamma$ has a continuous action on $\Hom(H, \Gamma)$ by conjugation. If $\phi\in \Hom(H, \Gamma)$ extends to a group homomorphism $\phi': G\rightarrow \Gamma$, then one has $g\phi=\phi'(g)\phi$ for every $g\in G$, whence the orbit closure of $\phi$ under the $G$-action coincides with the orbit closure of $\phi$ under the action of $\phi'(G)$, which in turn coincides with the orbit $X$ of $\phi$ under the action of $\overline{\phi'(G)}$.  Then the $G$-action on $X$ goes through the action of the compact group 
$\overline{\phi'(G)}$ on $X$ via the group homomorphism $\phi': G\rightarrow \overline{\phi'(G)}$. Note that  a continuous action of $G$ on a compact metrizable space is equicontinuous if and only if it goes through a continuous action of a compact Hausdorff group. Thus the $G$-action on $X$ is equicontinuous, and particularly, distal. 

In particular, for every compact metrizable group $\Gamma$, every $\phi\in \Hom(G, \Gamma)$ is $G$-point-distal. 
\end{remark} 

\begin{notation} \label{N-point-distal kernel}
Let $G$ be a countable discrete group. For a normal subgroup $H$ of $G$, we set
$$ H^{\{G\}}=\bigcap_\phi \ker \phi,$$
where $\phi$ ranges over $G$-point-distal homomorphisms of $H$ to compact metrizable groups. 
\end{notation}

For each $n\in \Nb$, denote by $U_n$ the group of complex unitary $n\times n$ matrices. 

\begin{lemma} \label{L-point-distal kernel}
Let $G$ be a countable discrete group and let $H$ be a normal subgroup of $G$. Then 
$$ H^{\{G\}}=\bigcap_\phi \ker \phi,$$
for $\phi$ ranging over $G$-point-distal homomorphisms of $H$ to $U_n$ for some $n\in \Nb$. 
\end{lemma}
\begin{proof} Denote by $H_1$ the right hand side of the equality. Clearly $H^{\{G\}}\subseteq H_1$. 

Let $h\in H\setminus H^{\{G\}}$. Then we can find a compact metrizable group $\Gamma$ and a $G$-point-distal homomorphism $\psi: H\rightarrow \Gamma$ such that $h\not\in \ker \psi$. Note that $\psi(h)$ is not the identity element of $\Gamma$. By the Peter-Weyl theorem \cite[Theorem II.3.1]{BT} we can find a continuous group homomorphism $\phi: \Gamma\rightarrow U_n$ for some $n\in \Nb$ such that $\psi(h)\not\in \ker \phi$. Note that $\phi$ induces a $G$-equivariant continuous map $\phi_*: \Hom(H, \Gamma)\rightarrow \Hom(H, U_n)$ sending $\psi'$ to $\phi\circ \psi'$. For any factor map $X\rightarrow Y$ between minimal continuous actions of $G$ on compact metrizable spaces, the images of distal points are distal \cite[Corollary IV.3.25.2]{deVries}. Write $X$ for the orbit closure of $\psi$ in $\Hom(H, \Gamma)$ under the $G$-action. Since $\psi$ is distal in $X$, by Corollary~\ref{minimal} and Proposition~\ref{P-minimal vs uniformly recurrent} the $G$-action on $X$ is minimal. Thus $\phi_*(\psi)=\phi\circ \psi$ is distal in $\phi_*(X)$. Therefore $\phi\circ \psi$ is $G$-point-distal. As $h\not\in \ker(\phi\circ \psi)$, we conclude that $h\not\in H_1$. Thus $H_1\subseteq H^{\{G\}}$, and hence $H_1=H^{\{G\}}$.
\end{proof}

\begin{lemma} \label{L-point-distal restriction}
Let $G$ be a countable discrete group and let $H_1\subseteq H_2$ be normal subgroups of $G$. Then 
$$ (H_1)^{\{G\}}\subseteq H_1\cap (H_2)^{\{G\}}.$$
\end{lemma}
\begin{proof} Let $\Gamma$ be a compact metrizable group. Then we have the restriction map $\pi: \Hom(H_2, \Gamma)\rightarrow \Hom(H_1, \Gamma)$. Clearly $\pi$ is $G$-equivariant and continuous. Arguing as in the proof of Lemma~\ref{L-point-distal kernel} we see that if $\phi\in \Hom(H_2, \Gamma)$ is $G$-point-distal, then so is $\pi(\phi)\in \Hom(H_1, \Gamma)$. It follows that $H_1\setminus (H_2)^{\{G\}}\subseteq H_1\setminus (H_1)^{\{G\}}$, equivalently, $(H_1)^{\{G\}}\subseteq H_1\cap (H_2)^{\{G\}}$.
\end{proof}

Let $G$ be a countable discrete group. We define a normal subgroup $H_\alpha$ of $G$ for each ordinal $\alpha$ transfinitely as follows. We start with 
$$H_0=G.$$
Once we have defined $H_\alpha$, we set
$$H_{\alpha+1}=(H_\alpha)^{\{G\}}.$$
For a limit ordinal $\beta$, if we have defined $H_\gamma$ for all $\gamma<\beta$, we set
$$H_\beta=\bigcap_{\gamma<\beta}H_\gamma.$$
Then $H_\alpha\supseteq H_\beta$ whenever $\alpha\le \beta$. There is a smallest countable ordinal $\kappa$ satisfying that $H_\kappa=H_{\kappa+1}$. We call $H_\kappa$ the {\it point-distal radical} of $G$. 

\begin{proposition} \label{P-point-distal radical}
Let $G$ be a countable discrete group. The point-distal radical of $G$ is the largest normal subgroup $H$ of $G$ satisfying $H^{\{G\}}=H$. 
\end{proposition}
\begin{proof} Let $H_\alpha$ and $\kappa$ be as in the definition of the point-distal radical. Then 
$$(H_\kappa)^{\{G\}}=H_{\kappa+1}=H_\kappa.$$

Let $H$ be a normal subgroup of $G$ satisfying $H^{\{G\}}=H$. We claim that $H\subseteq H_\alpha$ for each ordinal $\alpha$. Clearly $H\subseteq G=H_0$. If $H\subseteq H_\alpha$ for some $\alpha$, then by Lemma~\ref{L-point-distal restriction} we have
$$ H=H^{\{G\}}\subseteq H\cap (H_\alpha)^{\{G\}}= H\cap H_{\alpha+1},$$
whence $H\subseteq H_{\alpha+1}$. If $\alpha$ is a limit ordinal and $H\subseteq H_\gamma$ for all $\gamma<\alpha$, then
$$ H\subseteq \bigcap_{\gamma<\alpha}H_\gamma= H_\alpha.$$
This proves our claim. In particular, $H\subseteq H_\kappa$. 
\end{proof}

\begin{corollary} \label{C-point-distal radical}
Let $G$ be a countable discrete group. The following are equivalent:
\begin{enumerate}
\item For every normal subgroup $H$ of $G$ with $H\neq \{e_G\}$, one has $H^{\{G\}}\neq H$;
\item The point-distal radical of $G$ is $\{e_G\}$.
\end{enumerate}
\end{corollary}

\begin{remark} \label{R-distal radical and Bohr radical}
Let $G$ be a countable discrete group, let $H$ be a normal subgroup of $G$, and let $\Gamma$ be a compact metrizable group. We say $\phi\in \Hom(H, \Gamma)$ is {\it $G$-distal} if the $G$-action on the orbit closure of $\phi$ in $\Hom(H, \Gamma)$ is distal. We set $H^{(G)}$ as $\bigcap_\phi \ker \phi$ for $\phi$ ranging over $G$-distal homomorphisms of $H$ to compact metrizable groups. (This is \cite[Definition 2.4]{Abels83} in the case of countable discrete groups.) Then we define the {\it distal radical} of $G$ in the same way as the point-distal radical, replacing $(H_\alpha)^{\{G\}}$ by $(H_\alpha)^{(G)}$. The distal analogues of Lemma~\ref{L-point-distal restriction}, Proposition~\ref{P-point-distal radical}, and Corollary~\ref{C-point-distal radical} all hold. (The distal analogue of Corollary~\ref{C-point-distal radical} was mentioned in the last paragraph of \cite[page 175]{Abels83}.)

Write $\widetilde{H}$ for the kernel of the canonical homomorphism from $H$ to its Bohr compactification, i.e., $\bigcap_\phi \ker \phi$ for $\phi$ ranging over group homomorphisms of $H$ to compact metrizable groups.  Then we define the {\it Bohr radical} of $G$ in the same way as the point-distal radical, replacing $(H_\alpha)^{\{G\}}$ by $\widetilde{H_\alpha}$. The Bohr analogues of Lemma~\ref{L-point-distal restriction}, Proposition~\ref{P-point-distal radical}, and Corollary~\ref{C-point-distal radical} all hold.

Clearly, one has
$$\widetilde{H}\subseteq H^{\{G\}}\subseteq H^{(G)}$$
for every normal subgroup $H$ of $G$, whence
$$ \mbox{Bohr radical } \subseteq \mbox{ point-distal radical } \subseteq \mbox{ distal radical}.$$
From Remark~\ref{R-extension to distal} one has
$$\widetilde{G}=G^{\{G\}}=G^{(G)}.$$
\end{remark}

Let $\pi: X\rightarrow Y$ be a factor map between minimal continuous actions of $G$ on compact metrizable spaces $X$ and $Y$.  We say that $\pi$ is {\it almost $1$-$1$} \cite[Remark IV.6.1.2]{deVries}, if there is some $y\in Y$ with $|\pi^{-1}(y)|=1$.  We say that $\pi$ is {\it isometric} \cite[Remark V.5.20.2 and Corollary V.5.19]{deVries} if there is a continuous function $d: R_\pi\rightarrow [0, \infty)$, where $R_\pi=\{(x_1, x_2)\in X^2: \pi(x_1)=\pi(x_2)\}$, satisfying the following conditions:
\begin{enumerate}
\item for each $y\in Y$, the restriction of $d$ on $\pi^{-1}(y)\times \pi^{-1}(y)$ is a (compatible) metric on $\pi^{-1}(y)$;
\item $d$ is $G$-invariant in the sense that $d(gx_1, gx_2)=d(x_1, x_2)$ for all $(x_1, x_2)\in R_\pi$ and $g\in G$.
\end{enumerate}

We say that a minimal continuous action of $G$ on a compact metrizable space $X$ is {\it strictly AI} \cite[Remark VI.4.17.2]{deVries} if there is a transfinite sequence of  minimal continuous actions of $G$ on compact metrizable spaces $X_\beta$ for $0\le \beta\le \alpha$, where $\alpha$ is an ordinal, together with a factor map $\pi_{\gamma, \beta}: X_\beta\rightarrow X_\gamma$ for all $0\le \gamma\le \beta\le \alpha$ being compatible in the sense that $\pi_{\theta, \gamma}\circ \pi_{\gamma, \beta}=\pi_{\theta, \beta}$ for all $0\le \theta\le \gamma\le \beta\le \alpha$ and $\pi_{\beta, \beta}$ is the identity map on $X_\beta$ for all $0\le \beta\le \alpha$, so that the following hold:
\begin{enumerate}
\item for each $\beta<\alpha$, $\pi_{\beta, \beta+1}$ is either almost $1$-$1$ or isometric;
\item for each limit ordinal $\beta\le \alpha$, $X_\beta$ is the inverse limit of $X_\gamma$ for $\gamma<\beta$;
\item $X=X_\alpha$ and $X_0$ is a singleton. 
\end{enumerate} 
Note that, since $X=X_\alpha$ is metrizable, the set of $\beta<\alpha$ with $\pi_{\beta, \beta+1}$ failing to be injective is countable.  

The following lemma is due to Ellis \cite{Ellis73} \cite[VI.6.4.6]{deVries}. In fact $X$ has a dense $G_\delta$ subset consisting of distal points, but we do not need this fact here. 

\begin{lemma} \label{L-strictly AI are point-distal}
Let a countable discrete group $G$ act on  a compact metrizable space $X$ continuously and minimally. Assume that the action $G\curvearrowright X$ is strictly AI. Then $X$ has a distal point. 
\end{lemma}

\begin{lemma} \label{L-trivial action}
Let a countable discrete group $G$ act on  a compact metrizable space $X$ continuously and minimally. Assume that the action $G\curvearrowright X$ is strictly AI. Let $H$ be a normal subgroup of $G$ with $H^{\{G\}}=H$. Then the action of $H$ on $X$ is trivial. 
\end{lemma}
\begin{proof} Let $(X_\beta)_{0\le \beta\le \alpha}$ witness that $G\curvearrowright X$ is strictly AI. We shall show by transfinite induction on $\beta$ that the action of $H$ on $X_\beta$ is trivial for every $0\le \beta\le \alpha$. This is obvious when $\beta=0$. 

Assume that $\beta<\alpha$ and the action of $H$ on $X_\beta$ is trivial. Then $\pi_{\beta, \beta+1}$ is either almost $1$-$1$ or isometric. We separate the discussion into two cases. 

Case I: $\pi_{\beta, \beta+1}$ is almost $1$-$1$. We can find $x\in X_{\beta}$ such that $|\pi_{\beta, \beta+1}^{-1}(x)|=1$. Then $H$ fixes $\pi_{\beta, \beta+1}^{-1}(x)$. Note that $|\pi_{\beta, \beta+1}^{-1}(gx)|=1$ for every $g\in G$. Thus actually $H$ fixes  $\pi_{\beta, \beta+1}^{-1}(gx)=g\pi_{\beta, \beta+1}^{-1}(x)$ for every $g\in G$. Since the action $G\curvearrowright X_{\beta+1}$ is minimal, the $G$-orbit of $\pi_{\beta, \beta+1}^{-1}(x)$ is dense in $X_{\beta+1}$. It follows that the action of $H$ on $X_{\beta+1}$ is trivial.

Case II: $\pi_{\beta, \beta+1}$ is isometric. Let $d: R_{\pi_{\beta, \beta+1}}\rightarrow [0, +\infty)$ witness that $\pi_{\beta, \beta+1}$ is isometric. 
For any $x_1, x_2\in X_\beta$, denote by $I(x_1, x_2)$ the set of $t|_{\pi_{\beta, \beta+1}^{-1}(x_1)}$ where $t$ is in the Ellis semigroup $E(X_{\beta+1})$ of $G\curvearrowright X_{\beta+1}$ and satisfies $t(\pi_{\beta, \beta+1}^{-1}(x_1))\subseteq \pi_{\beta, \beta+1}^{-1}(x_2)$. Fix a point $x_0\in X_\beta$, which we shall determine soon. Set 
$$I(x_0, X_{\beta+1})=\bigcup_{x\in X_\beta}I(x_0, x),$$
endowed with the topology of uniform convergence as a space of maps from $\pi_{\beta, \beta+1}^{-1}(x_0)$ to $X_{\beta+1}$. From \cite[Proposition 1.4]{Abels83} we know that the following hold:
\begin{enumerate}
\item for any $x_1, x_2\in X_\beta$, the set $I(x_1, x_2)$ is nonempty, each $\phi\in I(x_1, x_2)$ is a bijection from $\pi_{\beta, \beta+1}^{-1}(x_1)$ to $\pi_{\beta, \beta+1}^{-1}(x_2)$, and $\phi^{-1}$ lies in $I(x_2, x_1)$;
\item $I(x_0, X_{\beta+1})$ is a compact metrizable space;
\item there is a natural $G$-action on $I(x_0, X_{\beta+1})$ given by
$$g(t|_{\pi_{\beta, \beta+1}^{-1}(x_0)})=(gt)|_{\pi_{\beta, \beta+1}^{-1}(x_0)},$$
where $g\in G$ and $t\in E(X_{\beta+1})$. The $G$-action on $I(x_0, X_{\beta+1})$ is continuous and minimal;
\item the map $q: I(x_0, X_{\beta+1})\rightarrow X_\beta$ sending $\varphi$ to $\pi_{\beta, \beta+1}(\im \varphi)$ is a factor map;
\item the map $R_q=\{(\varphi, \phi)\in I(x_0, X_{\beta+1})^2: q(\varphi)=q(\phi)\}\rightarrow I(x_0, x_0)$ sending $(\varphi, \phi)$ to $\varphi^{-1}\phi$ is continuous. 
\end{enumerate}

By Lemma~\ref{L-strictly AI are point-distal} we may choose $x_0$ to be  distal in $X_\beta$.
Denote by ${\rm Id}(x_0)$ the identity map of $\pi_{\beta, \beta+1}^{-1}(x_0)$. Then 
$${\rm Id}(x_0)\in I(x_0, x_0)\subseteq I(x_0, X_{\beta+1}).$$ 
We claim that ${\rm Id}(x_0)$ is distal in $I(x_0, X_{\beta+1})$ under the $G$-action. Let $\varphi\in I(x_0, X_{\beta+1})$ such that $({\rm Id}(x_0), \varphi)$ is proximal. Then $(q({\rm Id}(x_0)), q(\varphi))=(x_0, q(\varphi))$ is proximal. Since $x_0$ is distal in $X_\beta$, we have $q(\varphi)=x_0$. Thus $\varphi\in I(x_0, x_0)$. Since $({\rm Id}(x_0), \varphi)$ is proximal, there are some $\psi\in I(x_0, X_{\beta+1})$ and  a net $(g_j)$ in $G$ such that both $g_j{\rm Id}(x_0)$ and $g_j\varphi$ converge to $\psi$ as $j\to \infty$. Passing to a subset if necessary, we may assume that $g_j$ converges to some $t\in E(X_{\beta+1})$ as $j\to \infty$. Then for each $y\in \pi_{\beta, \beta+1}^{-1}(x_0)$ we have
$$ t({\rm Id}(x_0)(y))=\psi(y)=t(\varphi(y)).$$
We have $t(\pi_{\beta, \beta+1}^{-1}(x_0))\subseteq \pi_{\beta, \beta+1}^{-1}(x_1)$ for some $x_1\in X_\beta$. 
Since $t|_{\pi_{\beta, \beta+1}^{-1}(x_0)}\in I(x_0, x_1)$ is injective as a map from $\pi_{\beta, \beta+1}^{-1}(x_0)$ to $\pi_{\beta, \beta+1}^{-1}(x_1)$, we conclude that ${\rm Id}(x_0)=\varphi$. This proves our claim.

Recall that the restriction of $d$ on $\pi_{\beta, \beta+1}^{-1}(x)\times \pi_{\beta, \beta+1}^{-1}(x)$ is a compatible metric, denoted by $d_x$, on $\pi_{\beta, \beta+1}^{-1}(x)$ for each $x\in X_\beta$. It is easy to check that for any $x_1, x_2\in X_\beta$, every $\phi\in I(x_1, x_2)$ is an isometry from $\pi_{\beta, \beta+1}^{-1}(x_1)$ to $\pi_{\beta, \beta+1}^{-1}(x_2)$.

Denote by ${\rm Isom}(x_0, d)$ the isometry group of $\pi_{\beta, \beta+1}^{-1}(x_0)$ under $d_{x_0}$. This is a compact metrizable group endowed with the topology of uniform convergence. Let
$$\sigma: H\rightarrow {\rm Isom}(x_0, d)$$
be the group homomorphism by restricting the maps of $H$ to $\pi_{\beta, \beta+1}^{-1}(x_0)$. We define a map 
$$\Psi: I(x_0, X_{\beta+1})\rightarrow \Hom(H, {\rm Isom}(x_0, d))$$
by
$$ \Psi(\phi)(h)=\phi^{-1}h\phi$$
for $\phi\in I(x_0, X_{\beta+1})$ and $h\in H$. Clearly $\Psi$ is $\Gamma$-equivariant and 
$$\Psi({\rm Id}(x_0))=\sigma.$$
For each fixed $h\in H$, by (3) and (5) the map $I(x_0, X_{\beta+1})\rightarrow {\rm Isom}(x_0, d)$ sending $\phi$ to $\phi^{-1}h\phi$ is continuous. Thus $\Psi$ is continuous. Therefore $\Psi$ is a factor map from $I(x_0, X_{\beta+1})$ to $\im \Psi$. 
For any factor map $X'\rightarrow Y'$ between minimal continuous actions of $G$ on compact metrizable spaces, the images of distal points are distal \cite[Corollary IV.3.25.2]{deVries}. As the $G$-action on $I(x_0, X_{\beta+1})$ is minimal and ${\rm Id}(x_0)$ is distal in $I(x_0, X_{\beta+1})$, we conclude that $\Psi({\rm Id}(x_0))=\sigma$ is distal in $\im \Psi$. Thus $\sigma$ is $G$-point-distal.
As $H^{\{G\}}=H$, we see that the action of $H$ on $\pi_{\beta, \beta+1}^{-1}(x_0)$ is trivial. 
Since $H$ is normal in $G$, $H$ fixes $gy$ for every $y\in \pi_{\beta, \beta+1}^{-1}(x_0)$ and $g\in G$. 
Because the $G$-action on $X_{\beta+1}$ is minimal, we get that the action of $H$ on $X_{\beta+1}$ is trivial. 

This finishes the discussion in Case II. 

Assume that $\beta\le \alpha$ is a limit ordinal and that the action of $H$ on $X_\gamma$ is trivial for all $\gamma<\beta$. Since the $H$-action on $X_\beta$ is the inverse limit of the $H$-actions on $X_\gamma$ for $\gamma<\beta$, we see that the $H$-action on $X_\beta$ is also trivial.

This finishes the transfinite induction. Now we get that the action of $H$ on $X=X_\alpha$ is trivial.
 \end{proof}

The following is the point-distal analogue of the discrete case of \cite[Theorem 2.5]{Abels83}. 

\begin{theorem} \label{P-effective to trivial radical}
Let $G$ be a countable discrete group. If the set of distal points in $(2^G, G)$ is dense, then the point-distal radical of $G$ is $\{e_G\}$.
\end{theorem}
\begin{proof} By Theorem~\ref{P-effective point-distal} and Corollary~\ref{C-point-distal radical}, it suffices to show that if $H$ is a normal subgroup of $G$ satisfying $H^{\{G\}}=H$, then $H$ acts on $X$ trivially for every minimal continuous action of $G$ on a compact metrizable space $X$ with a distal point. By the Veech structure theorem \cite{Veech2} (also see \cite[Theorem VI.4.26]{deVries}) there is a minimal continuous action of $G$ on a compact metrizable space $Y$ such that the action $G\curvearrowright Y$ is strictly AI and is an almost $1$-$1$ extension of $G\curvearrowright X$. By Lemma~\ref{L-trivial action} the action of $H$ on $Y$ is trivial. Thus the action of $H$ on $X$ is trivial. 
\end{proof}

\begin{question} \label{Q-effective vs trivial radical}
Does the converse of Theorem~\ref{P-effective to trivial radical} hold? 
\end{question}

\begin{example}\label{G_2}
Let $\mathbb{Q}^*$ be the group of nonzero rational numbers under multiplication. The group $\mathbb{Q}^*$ acts on the additive group $\Qb$ by group automorphisms via multiplication. Denote by $G_2$ the semidirect product $\mathbb{Q}\rtimes \mathbb{Q}^*$. It is easily checked that $G_2$ is 2-step nilpotent. The proof of \cite[Example 4.3]{Abels83} shows that the point-distal radical of $G_2$ is $\mathbb{Q}$. Since it is nontrivial, we conclude from Theorem~\ref{P-effective to trivial radical} that the set of distal points in $(2^{G_2}, G_2)$ is not dense. Since $G_1$ in Example~\ref{E-effective point-distal action} is a subgroup of $G_2$, we have that $G_2$ is not MAP.
\end{example}

\section{Some Related Results\label{sec:5}}

In this final section we present two results regarding denseness of distal points. The first is to consider the collection of all countable groups $G$ for which the distal points in $(2^G, G)$ are dense, and to show that this collection is closed under taking 
finite-index extensions. The second is to consider an extreme opposite of denseness of distal points -- groups $G$ for which there are only two trivial distal points in $(2^G, G)$.

\subsection{Finite-index extensions}\hfill

Denote by $\mathcal{G}$ the class of countable groups $G$ such that the set of distal points in $(2^G,G)$ is dense. By our results in Section~\ref{sec:3} and from Example~\ref{E-effective point-distal action}, the class of countable MAP groups is a proper subclass of $\mathcal{G}$. 

We make the following observations regarding closure properties of $\mathcal{G}$. If $A$ is a TIP*-set of a countable group $G$ and $H$ is a subgroup of $G$, then by definition $A\cap H$ is a TIP*-set of $H$. By Lemma \ref{lem:sIP*}, the class $\mathcal{G}$ is closed under taking subgroups. 

The free group $\mathbb{F}_\omega$ with countably infinitely many generators is residually finite. Since every countable group is a quotient group of $\mathbb{F}_\omega$, $\mathcal{G}$ is not closed under taking quotients. 

By Example \ref{G_2}, $\mathcal{G}$ is not closed under group extension. 

The next theorem shows that $\mathcal{G}$ is closed under finite-index extension.

\begin{theorem} Let $G$ be a countably infinite discrete group. Suppose $H\leq G$ has finite index and the set of all distal points is dense in $(2^H, H)$. Then the set of all distal points is dense in $(2^G, G)$.
\end{theorem}

\begin{proof} Let $F\subseteq G$ be a finite subset and $p: F\to \{0,1\}$. We may assume $e_G\in F$. Let $a_0=e_G, a_1, \dots, a_k$ be a complete set of left-coset representatives for $H$ in $G$. By choosing $a_1,\dots, a_k$ appropriately and possibly extending $F$ and $p$, we may assume $a_i\in F$ for all $0\leq i\leq k$ and $p(a_i)=p(e_G)$ for all $0\leq i\leq k$. We define an IP*-recurrent point $x\in 2^G$ with $x|_F=p$. For $0\leq i\leq k$, let $F_i=a_i^{-1}(F\cap a_iH)\subseteq H$ and let $p_i: F_i\to \{0,1\}$ be defined as $p_i(h)=p(a_ih)$ for $h\in F_i$. Since the set of all distal points is dense in $(2^H, H)$, there exist IP*-recurrent points $x_i\in 2^H$ with $x_i|_{F_i}=p_i$. Let $x\in 2^G$ be defined as $x(a_ih)=x_i(h)$ for all $h\in H$. Obviously $x|_{F}=p$. In particular, $x(a_i)=x_i(e_G)=p_i(e_G)=p(a_i)=p(e_G)=x(e_G)$ for all $0\leq i\leq k$. We verify that $x$ is ${\rm IP}^*$-recurrent.

It suffices to show that $A=\{g\in G\,:\, x(g)=x(e_G)\}$ is a TIP*-set. For this, let $\gamma\in G$. Suppose $\gamma^{-1}=a_{i_0}\eta^{-1}$ where $0\leq i_0\leq k$ and $\eta\in H$. We need to show that either $\gamma A$ or $G\setminus \gamma A$ is an ${\rm IP}^*$-set.

Let $B=a_{i_0}^{-1}(A\cap a_{i_0}H)$. Then $B\subseteq H$, and for any $h\in H$,
$$ h\in B\iff a_{i_0}h\in A \iff x(a_{i_0}h)=x(e_G) \iff x_{i_0}(h)=x_{i_0}(e_G). $$
Since $x_{i_0}\in 2^H$ is ${\rm IP}^*$-recurrent, $B$ is a TIP*-set on $H$, and in particular either $\eta B$ or $H\setminus \eta B$ is an IP*-set on $H$. Without loss of generality assume $\eta B$ is an IP*-set on $H$. The argument for the case $H\setminus \eta B$ is an IP*-set is similar.

Noting that $\eta B=\gamma A\cap H$, we just need to show that $\eta B$ is an IP*-set on $G$. Let $N$ be an IP-set in $G$. By Proposition \ref{Ramsey}, there is $0\le i\le k$ such that $N\cap a_iH$ contains an IP-set on $G$. Note that for any $1\le j\le k$ and $g_1,g_2\in a_jH$, we have $g_1g_2\notin a_jH$; thus, by the definition of IP-sets, $a_jH$ contains no IP-sets on $G$. It follows that $N\cap a_0H=N\cap H$ contains an IP-set on $G$. Since the set $\eta B$ is an IP*-set on $H$, we have $\eta B\cap N=\eta B\cap N\cap H\ne\varnothing$. This shows that $\eta B$ is an IP*-set on $G$. 
\end{proof}

\subsection{Minimally almost periodic groups}\hfill

In this subsection we consider countably infinite groups $G$ for which the system $(2^G,G)$ just has two trivial distal points $\overline{0}$ (the constant $0$ sequence) and $\overline{1}$ (the constant $1$ sequence). This is the opposite of the denseness of distal points. It turns out that this also corresponds to a well-studied notion for topological groups. 

\begin{definition} A topological group $G$ is {\em minimally almost periodic} (or {\em MinAP}) if it has no non-trivial continuous homomorphism to any compact Hausdorff topological group.
\end{definition}

\begin{theorem} 
Let $G$ be  a countable discrete group. The following are equivalent:
\begin{enumerate}
\item $G$ is minAP;
\item The system $(2^G,G)$ has only two trivial distal points $\overline{0}$ and $\overline{1}$;
\item The system $(2^G,G)$ has only two trivial almost automorphic points $\overline{0}$ and $\overline{1}$.
\end{enumerate}
\end{theorem}
\begin{proof}
(1)$\Rightarrow$(2): If there is a nontrivial distal point $x\in 2^G$, then $(\overline{Gx},G)$ is point-distal. By \cite[Theorem 6.1]{Veech2}, there is a non-trivial minimal equicontinuous system $(Y,G)$. Let $\varphi:G\to{\rm Homeo}(Y)$ denote the action of $G$ on $Y$. The enveloping semigroup $E(Y)$ of $(Y,G)$ is a compact group and $\varphi(G)$ is dense in $E(Y)$ (see \cite[Chapter 3, Theorem 3]{AusBook}). In particular, $\varphi$ is a non-trivial group homomorphism from $G$ to some compact Hausdorff topological group, contradicting (1).

(2)$\Rightarrow$(3): Every almost automorphic point is distal.

(3)$\Rightarrow$(1): If $G$ is not minAP, then there is a non-trivial homomorphism $\pi$ from $G$ to some compact Hausdorff topological group $\Gamma$. Take $g\in G$ such that $\pi(g)\ne e_\Gamma$. The image $\pi(G)$ is MAP, from which it follows that there is a T$\Delta^*$-set $A$ of $\pi(G)$ such that $e_\Gamma\in A$ and $\pi(g)\notin A$. It is easily checked that $\pi^{-1}(A)$ is a T$\Delta^*$-set of $G$, $g\notin \pi^{-1}(A)$ and $e_G\in \pi^{-1}(A)$. Then define $x\in 2^G$ such that $x(h)=1$ if and only if $h\in \pi^{-1}(A)$, $x$ is a non-trivial almost automorphic point, contradicting (3).
\end{proof}

It was shown in \cite{XY} that the Tarski monster group is MinAP, and hence it has only trivial distal points $\overline{0}$ and $\overline{1}$. Some other examples of countable discrete MinAP groups are
\begin{enumerate}
\item[(i)] the finitely-supported alternating group $\mbox{\rm Alt}(\mathbb{N})$;
\item[(ii)] Hall's universal countable locally finite group $\mathbb{H}$;
\item[(iii)] $SL(n,K)$ and $PSL(n,K)$ when $n\geq 2$ and $K$ is a countably infinite field.
\end{enumerate}
For each of them, the system $(2^G,G)$ has only two trivial distal points $\overline{0}$ and $\overline{1}$.

\end{document}